\documentclass[a4paper, oneside]{article}
\usepackage{enumitem}
\usepackage{lmodern}
\usepackage{amssymb,amsthm,amsmath}
\usepackage[utf8]{inputenc}
\usepackage{verbatim}
\usepackage{tikz}
\usepackage{authblk}
\usepackage{cite}
\usetikzlibrary{decorations.pathreplacing}

\theoremstyle{definition}
\newtheorem{dfntn}{Definition}[section]
\newtheorem{lmm}[dfntn]{Lemma}
\newtheorem{prpstn}[dfntn]{Proposition}
\newtheorem{thrm}[dfntn]{Theorem}
\newtheorem{xmpl}[dfntn]{Example}
\newtheorem{crllr}[dfntn]{Corollary}
\newtheorem{rmrk}[dfntn]{Remark}
\newtheorem{prblm}[dfntn]{Problem}

\newcommand{\N}{\mathbb{N}}
\newcommand{\R}{\mathbb{R}}
\newcommand{\Z}{\mathbb{Z}}
\newcommand{\abs}[1]{\vert #1 \vert} 
\DeclareMathOperator{\config}{config} 
\DeclareMathOperator{\real}{real} 
\DeclareMathOperator{\integ}{int} 
\DeclareMathOperator{\cyl}{Cyl} 
\DeclareMathOperator{\tr}{Tr} 

\begin{document}

\title{Cellular Automata and Powers of $p/q$ \thanks{The work was partially supported by the Academy of Finland grant 296018 and by the Vilho, Yrjö and Kalle Väisälä Foundation}}

\author{Jarkko Kari}
\author{Johan Kopra}

\affil{Department of Mathematics and Statistics, \\FI-20014 University of Turku, Finland}
\affil{jkari@utu.fi, jtjkop@utu.fi}

\date{}

\maketitle

\setcounter{page}{1}

\begin{abstract} We consider one-dimensional cellular automata $F_{p,q}$ which multiply numbers by $p/q$ in base $pq$ for relatively prime integers $p$ and $q$. By studying the structure of traces with respect to $F_{p,q}$ we show that for $p\geq 2q-1$ (and then as a simple corollary for $p>q>1$) there are arbitrarily small finite unions of intervals which contain the fractional parts of the sequence $\xi(p/q)^n$, ($n=0,1,2,\dots$) for some $\xi>0$. To the other direction, by studying the measure theoretical properties of $F_{p,q}$, we show that for $p>q>1$ there are finite unions of intervals approximating the unit interval arbitrarily well which don't contain the fractional parts of the whole sequence $\xi(p/q)^n$ for any $\xi>0$. \end{abstract}

\providecommand{\keywords}[1]{\textbf{Keywords:} #1}
\noindent\keywords{distribution modulo 1, Z-numbers, cellular automata, ergodicity, strongly mixing}

\section*{Introduction}
In \cite{W} Weyl proved that for any $\alpha>1$ the sequence of numbers $\{\xi\alpha^i\}$, $i\in\N$ is uniformly distributed in the interval $[0,1)$ for almost every choice of $\xi>0$, where $\{x\}=x-\lfloor x \rfloor$ is the fractional part of $x$. In particular, $\{\{\xi\alpha^i\}\mid i\in\N\}$ is dense in $[0,1)$ for almost every $\xi>0$. However, this doesn't hold for every $\xi>0$, and it would be interesting to know what other types of distribution the set $\{\{\xi\alpha^i\}\mid i\in\N\}$ can exhibit for different choices of $\xi$.

As a special case of this problem, in \cite{M} Mahler posed the question of whether there exist so called $Z$-numbers, i.e. real numbers $\xi>0$ such that
\[\left\{\xi\left(\frac{3}{2}\right)^i\right\}\in [0,1/2)\]
for every $i\in\N$. We will work with the following generalization of the notion of $Z$-numbers: let $p>q>1$ be relatively prime integers and let $S\subseteq[0,1)$ be a finite union of intervals. Then if we denote by $Z_{p/q}(S)$ the set of real numbers $\xi>0$ such that 
\[\left\{\xi\left(\frac{p}{q}\right)^i\right\}\in S\]
for every $i\in\N$, $Z$-numbers are the elements of the set $Z_{3/2}([0,1/2))$ and Mahler's question can be reformulated as whether $Z_{3/2}([0,1/2))=\emptyset$ or not.

A natural approach to the emptiness problem of $Z_{3/2}([0,1/2))$ is to seek sets $S$ as small as possible such that $Z_{p/q}(S)\neq\emptyset$ and sets $S$ as large as possible such that $Z_{p/q}(S)=\emptyset$ (for previous results, see e.g. \cite{A,AFS,D,FLP}). In this paper we prove that for $p\geq 2q-1$ and $k>0$ there exists a union of $q^{2k}$ intervals $I_{p,q,k}$ of total length at most $(q/p)^k$ such that $Z_{p/q}(I_{p,q,k})$ is non-empty. From this it follows as a simple corollary that for $p>q$ and $\epsilon>0$ there exists a finite union of intervals $J_{p,q,\epsilon}$ of total length at most $\epsilon$ such that $Z_{p/q}(J_{p,q,\epsilon})$ is non-empty. On the other hand, for $p>q$ and $\epsilon>0$ we prove that there exists a finite union of intervals $K_{p,q,\epsilon}$ of total length at least $1-\epsilon$ such that $Z_{p/q}(K_{p,q,\epsilon})$ is empty. The proofs of emptiness and non-emptiness are based on the study of the cellular automaton $F_{p,q}$ that implements multiplication by $p/q$ in base $pq$. This cellular automaton was introduced in \cite{K2} in relation with the problem of universal pattern generation and the connection to Mahler's problem was pointed out in \cite{K}.

\section{Preliminaries}

For a finite set $A$ (an \emph{alphabet}) the set $A^{\Z}$ is called a \emph{configuration space} and its elements are called \emph{configurations}. An element $c\in A^{\Z}$ is a bi-infinite sequence and the element at position $i$ in the sequence is denoted by $c(i)$. A \emph{factor} of $c$ is any finite sequence $c(i) c(i-1)\dots c(j)$ where $i,j\in\Z$, and we interpret the sequence to be empty if $j<i$. Any finite sequence $a(1) a(2)\dots a(n)$ (also the empty sequence, which is denoted by $\lambda$) where $a(i)\in A$  is a \emph{word} over $A$. The set of all words over $A$ is denoted by $A^*$, and the set of non-empty words is $A^+=A^*\setminus\{\lambda\}$. The set of words of length $n$ is denoted by $A^n$. For a word $w\in A^*$, $\abs{w}$ denotes its length, i.e. $\abs{w}=n\iff w\in A^n$.

\begin{dfntn}Any $w\in A^+$ and $i\in\Z$ determine a \emph{cylinder}
\[\cyl_A(w,i)=\{c\in A^{\Z}\mid c(i)c(i+1)\dots c(i+\abs{w}-1)=w\}.\]
The collection of all cylinders over $A$ is
\[\mathcal{C}_A=\{\cyl_A(w,i)\mid w\in A^+,i\in\Z\}.\]
The subscript $A$ is omitted when the used alphabet is clear from the context.
\end{dfntn}

The configuration space $A^\Z$ becomes a topological space when endowed with the topology $\mathcal{T}$ generated by $\mathcal{C}$. It can be shown that this topology is metrizable, and that a set $S\subseteq A^\Z$ is compact if and only if it is closed. $A^{\Z}$ can also be endowed with measure theoretical structure: it is known that there is a unique probability space $(A^{\Z},\Sigma(\mathcal{C}),\mu)$, where $\Sigma(\mathcal{C})$ is the sigma-algebra generated by $\mathcal{C}$ and $\mu:\Sigma(\mathcal{C})\to\R$ is a measure such that $\mu(\cyl(w,i))=\abs{A}^{-\abs{w}}$ for every $\cyl(w,i)\in\mathcal{C}$. Note that $\mathcal{T}\subseteq\Sigma(\mathcal{C})$ because $\mathcal{C}$ is a countable basis of $\mathcal{T}$, so $\Sigma(\mathcal{C})$ is actually the collection of Borel sets of $A^\Z$.

\begin{sloppypar}
\begin{dfntn}A one-dimensional cellular automaton (CA) is a $3$-tuple $(A,N,f)$, where $A$ is a finite \emph{state set}, $N=(n_1,\dots,n_m)\in\Z^m$ is a \emph{neighborhood vector} and $f:A^m\to A$ is a \emph{local rule}. A given CA $(A,N,f)$ is customarily identified with a corresponding \emph{CA function} $F:A^{\Z}\to A^{\Z}$ defined by
\[F(c)(i)=f(c(i+n_1),\dots,c(i+n_m))\]
for every $c\in A^{\Z}$ and $i\in\Z$.\end{dfntn}
\end{sloppypar}

To every configuration space $A^{\Z}$ is associated a \emph{(left) shift} CA $(A,(1),\iota)$, where $\iota:A\to A$ is the identity function. Put in terms of the CA-function determined by this $3$-tuple, the left shift is $\sigma_A:A^{\Z}\to A^{\Z}$ defined by $\sigma_A(c)(i)=c(i+1)$ for every $c\in A^{\Z}$ and $i\in\Z$.

For a given CA $F:A^\Z\to A^\Z$ and a configuration $c\in A^\Z$ it is often helpful to consider a \emph{space-time diagram} of $c$ with respect to $F$. A space-time diagram is a picture which depicts elements of the sequence $(F^i(c))_{i\in\N}$ (or possibly $(F^i(c))_{i\in\Z}$ in the case when $F$ is reversible) in such a way that $F^{i+1}(c)$ is drawn below $F^i(c)$ for every $i$. As an example, Figure \ref{std} contains a space-time diagram of $c=\dots01101001\dots$ with respect to the left shift on $A=\{0,1\}$.

All CA-functions are continuous with respect to $\mathcal{T}$ and commute with the shift.

\begin{figure}[h]
\centering
\begin{tikzpicture}
\draw (0pt,60pt) -- (200pt,60pt);
\draw (0pt,40pt) -- (200pt,40pt);
\draw (0pt,20pt) -- (200pt,20pt);
\draw (0pt,0pt) -- (200pt,0pt);
\node[anchor=east] at (-5pt,50pt) {$c$};
\node at (10pt,50pt) {$\dots$}; \node at (30pt,50pt) {$0$}; \node at (50pt,50pt) {$1$}; \node at (70pt,50pt) {$1$}; \node at (90pt,50pt) {$0$};
\node at (110pt,50pt) {$1$}; \node at (130pt,50pt) {$0$}; \node at (150pt,50pt) {$0$}; \node at (170pt,50pt) {$1$}; \node at (190pt,50pt) {$\dots$};
\node[anchor=east] at (-5pt,30pt) {$\sigma_A(c)$};
\node at (10pt,30pt) {$\dots$}; \node at (30pt,30pt) {$1$}; \node at (50pt,30pt) {$1$}; \node at (70pt,30pt) {$0$}; \node at (90pt,30pt) {$1$};
\node at (110pt,30pt) {$0$}; \node at (130pt,30pt) {$0$}; \node at (150pt,30pt) {$1$};                          \node at (190pt,30pt) {$\dots$};
\node[anchor=east] at (-5pt,10pt) {$\sigma_A^2(c)$};
\node at (10pt,10pt) {$\dots$}; \node at (30pt,10pt) {$1$}; \node at (50pt,10pt) {$0$}; \node at (70pt,10pt) {$1$}; \node at (90pt,10pt) {$0$};
\node at (110pt,10pt) {$0$}; \node at (130pt,10pt) {$1$};                                                       \node at (190pt,10pt) {$\dots$};
\end{tikzpicture}
\caption{An example of a space-time diagram.}
\label{std}
\end{figure}

\section{The cellular automata $G_{p,q}$ and $F_{p,q}$}

In this section we define auxiliary CA $G_{p,q}$ for relatively prime $p,q\geq2$ and show that they multiply numbers by $p$ in base $pq$. Then we use $G_{p,q}$ in constructing the CA $F_{p,q}$ which multiply numbers by $p/q$ in base $pq$, and cover some basic properties of $F_{p,q}$.

Let us denote by $A_n$ the set of digits $\{0,1,2,\dots,n-1\}$ for $n\in\N$, $n>1$. To perform multiplication using a CA we need be able to represent a nonnegative real number as a configuration in $A_n^{\Z}$. If $\xi\geq0$ is a real number and $\xi=\sum_{i=-\infty}^{\infty}{\xi_i n^i}$ is the unique base $n$ expansion of $\xi$ such that $\xi_i\neq n-1$ for infinitely many $i<0$, we define $\config_n(\xi)\in A_n^{\Z}$ by
\[\config_n(\xi)(i)=\xi_{-i}\]
for all $i\in\Z$. In reverse, whenever $c\in A_n^{\Z}$ is such that $c(i)=0$ for all sufficiently small $i$, we define
\[\real_n(c)=\sum_{i=-\infty}^{\infty}{c(-i) n^i}.\]
For words $w=w(1)w(2)\dots w(k)\in A_n^k$ we define analogously
\[\real_n(w)=\sum_{i=1}^{k}{w(i)n^{-i}}.\]
Clearly $\real_n(\config_n(\xi))=\xi$ and $\config_n(\real_n(c))=c$ for every $\xi\geq0$ and every $c\in A_n^{\Z}$ such that $c(i)=0$ for all sufficiently small $i$ and $c(i)\neq n-1$ for infinitely many $i>0$.

For relatively prime integers $p,q\geq2$ let $g_{p,q}:A_{pq}\times A_{pq}\to A_{pq}$ be defined as follows. Digits $x,y\in A_{pq}$ are represented as $x=x_1q+x_0$ and $y=y_1q+y_0$, where $x_0,y_0\in A_q$ and $x_1,y_1\in A_p$: such representations always exist and they are unique. Then
\[g_{p,q}(x,y)=g_{p,q}(x_1q+x_0,y_1q+y_0)=x_0p+y_1.\]
An example in the particular case $(p,q)=(3,2)$ is given in Figure \ref{taul}.

\begin{figure}[h]
\centering
\begin{tabular} {c | c c c c c c}
$x\backslash y$ & 0 & 1 & 2 & 3 & 4 & 5 \\ \hline
0 & 0 & 0 & 1 & 1 & 2 & 2 \\
1 & 3 & 3 & 4 & 4 & 5 & 5 \\
2 & 0 & 0 & 1 & 1 & 2 & 2 \\
3 & 3 & 3 & 4 & 4 & 5 & 5 \\
4 & 0 & 0 & 1 & 1 & 2 & 2 \\
5 & 3 & 3 & 4 & 4 & 5 & 5 \\
\end{tabular}
\caption{The values of $g_{p,q}(x,y)$ in the case $(p,q)=(3,2)$.}
\label{taul}
\end{figure}

The CA function $G_{p,q}:A_{pq}^{\Z}\to A_{pq}^{\Z}$, $G_{p,q}(c)(i)=g_{p,q}(c(i),c(i+1))$ determined by $(A_{pq},(0,1),g_{p,q})$ implements multiplication by $p$ in base $pq$ in the sense of the following lemma.

\begin{lmm}\label{vastaavuus}$\real_{pq}(G_{p,q}(\config_{pq}(\xi)))=p\xi$ for all $\xi\geq 0$.\end{lmm}
\begin{proof}
Let $c=\config_{pq}(\xi)$. For every $i\in\Z$, denote by $c(i)_0$ and $c(i)_1$ the natural numbers such that $0\leq c(i)_0<q$, $0\leq c(i)_1<p$ and $c(i)=c(i)_1q+c(i)_0$. Then
\begin{flalign*}
&\real_{pq}(G_{p,q}(\config_{pq}(\xi)))=\real_{pq}(G_{p,q}(c))=\sum_{i=-\infty}^{\infty}G_{p,q}(c)(-i)(pq)^i \\
=&\sum_{i=-\infty}^{\infty}g_{p,q}(c(-i),c(-i+1))(pq)^i=\sum_{i=-\infty}^{\infty}(c(-i)_0 p+c(-i+1)_1)(pq)^i \\
=&\sum_{i=-\infty}^{\infty}(c(-i)_0 p(pq)^i+c(-i+1)_1pq(pq)^{i-1}) \\
=&\sum_{i=-\infty}^{\infty}(c(-i)_0 p(pq)^i+c(-i)_1pq(pq)^i) \\
=&p\sum_{i=-\infty}^{\infty}(c(-i)_1q+c(-i)_0)(pq)^i=p\real_{pq}(c)=p\real_{pq}(\config_{pq}(\xi))=p\xi.
\end{flalign*}
\end{proof}

We also define $G_{p,q}(w)$ for words $w=w(1)w(2)\dots w(\abs{w})$ such that $\abs{w}\geq 2$:
\[G_{p,q}(w)=u=u(1)\dots u(\abs{w}-1)\in A_{pq}^{\abs{w}-1},\]
where $u(i)=g_{p,q}(w(i),w(i+1))$ for $1\leq i\leq \abs{w}-1$.
Inductively it is possible to define $G_{p,q}^t(w)$ for every $t>0$ and word $w$ such that $\abs{w}\geq t+1$:
\[G_{p,q}^t(w)=G_{p,q}(G_{p,q}^{t-1}(w))\in A_{pq}^{\abs{w}-t}.\]

Clearly the shift CA $\sigma_{A_{pq}}$ multiplies by $pq$ in base $pq$ and its inverse divides by $pq$. This combined with Lemma \ref{vastaavuus} shows that the composition $F_{p,q}=\sigma_{A_{pq}}^{-1}\circ G_{p,q}\circ G_{p,q}$ implements multiplication by $p/q$ in base $pq$. The value of $F_{p,q}(c)(i)$ is given by the local rule $f_{p,q}$ defined as follows:
\begin{flalign*}
F_{p,q}(c)(i)=&f_{p,q}(c(i-1),c(i),c(i+1)) \\
=&g_{p,q}(g_{p,q}(c(i-1),c(i)),g_{p,q}(c(i),c(i+1))).
\end{flalign*}

The CA function $F_{p,q}$ is reversible: if $c\in A_{pq}^{\Z}$ is a configuration with a finite number of non-zero coordinates, then
\begin{flalign*}
&F_{p,q}(F_{q,p}(c))=F_{p,q}(F_{q,p}(\config_{pq}(\real_{pq}(c)))) \\
\overset{L \ref{vastaavuus}}{=}&\config_{pq}((p/q)(q/p)\real_{pq}(c))=c.
\end{flalign*}
Since $F_{p,q}\circ F_{q,p}$ is continuous and agrees with the identity function on a dense set, it follows that $F_{p,q}(F_{q,p}(c))=c$ for all configurations $c\in A_{pq}^{\Z}$. We will denote the inverse of $F_{p,q}$ interchangeably by $F_{q,p}$ and $F_{p,q}^{-1}$.

As for $G_{p,q}$, we define $F_{p,q}(w)$ for words $w=w(1)w(2)\dots w(\abs{w})$ such that $\abs{w}\geq 3$:
\[F_{p,q}(w)=u=u(1)\dots u(\abs{w}-2)\in A_{pq}^{\abs{w}-2},\]
where $u(i)=f_{p,q}(w(i),w(i+1),w(i+2))$ for $1\leq i\leq \abs{w}-2$, and $F_{p,q}^t(w)$ for every $t>0$ and word $w$ such that $\abs{w}\geq 2t+1$:
\[F_{p,q}^t(w)=F_{p,q}(F_{p,q}^{t-1}(w))\in A_{pq}^{\abs{w}-2t}\]
(see an example in Figure \ref{fonw}).

\begin{figure}
\centering
\begin{tikzpicture}
\draw (0pt,80pt) -- (140pt,80pt);
\draw (0pt,60pt) -- (140pt,60pt);
\draw (20pt,40pt) -- (120pt,40pt);
\draw (40pt,20pt) -- (100pt,20pt);
\draw (60pt,0pt) -- (80pt,0pt);

\draw (0pt,80pt) -- (0pt,60pt);		\draw (140pt,80pt) -- (140pt,60pt);
\draw (20pt,60pt) -- (20pt,40pt);		\draw (120pt,60pt) -- (120pt,40pt);
\draw (40pt,40pt) -- (40pt,20pt);		\draw (100pt,40pt) -- (100pt,20pt);
\draw (60pt,20pt) -- (60pt,0pt);		\draw (80pt,20pt) -- (80pt,0pt);

\node[anchor=west] at (-40pt,70pt) {$w$};
\node at (10pt,70pt) {$3$}; \node at (30pt,70pt) {$4$}; \node at (50pt,70pt) {$3$}; \node at (70pt,70pt) {$4$}; \node at (90pt,70pt) {$2$};
\node at (110pt,70pt) {$0$}; \node at (130pt,70pt) {$5$};
\node[anchor=west] at (-40pt,50pt) {$F_{3,2}(w)$};
\node at (30pt,50pt) {$3$}; \node at (50pt,50pt) {$5$}; \node at (70pt,50pt) {$3$}; \node at (90pt,50pt) {$3$};
\node at (110pt,50pt) {$1$};
\node[anchor=west] at (-40pt,30pt) {$F_{3,2}^{2}(w)$};
\node at (50pt,30pt) {$5$}; \node at (70pt,30pt) {$2$}; \node at (90pt,30pt) {$1$};
\node[anchor=west] at (-40pt,10pt) {$F_{3,2}^{3}(w)$};
\node at (70pt,10pt) {$0$};
\end{tikzpicture}
\caption{Iterated application of $F_{p,q}$ on $w$ for $(p,q)=(3,2)$ and $w=3434205$.}
\label{fonw}
\end{figure}
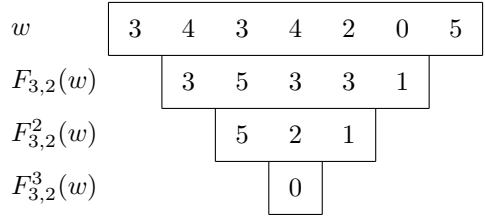

By the definition of $F_{p,q}$, for every $c\in A_{pq}^{\Z}$ and every $i\in\Z$ the value of $F_{p,q}(c)(i)$ is uniquely determined by $c(i-1),c(i)$ and $c(i+1)$, the three nearest digits above in the space-time diagram. Proposition \ref{leftdet} gives similarly that each digit in the space-time diagram is determined by the three nearest digits to the right (see Figure \ref{expdirs}). Its proof is broken down into the following sequence of lemmas.

\begin{figure}
\centering
\begin{tikzpicture}
\draw (0pt,60pt) -- (120pt,60pt);
\draw (0pt,30pt) -- (120pt,30pt);
\draw (0pt,0pt) -- (120pt,0pt);
\node (u1) at (20pt,45pt) {$c(i-1)$}; \node (u2) at (60pt,45pt) {$c(i)$}; \node (u3) at (100pt,45pt) {$c(i+1)$};
                                 \node (d) at (60pt,15pt) {$F_{p,q}(c)(i)$};
\draw[->,thick](u1.south) to [bend right = 45] (d.west);\draw[->,thick](u2) to (d);\draw[->,thick](u3.south) to [bend left = 45] (d.east);
																
\draw (180pt,90pt) -- (280pt,90pt);
\draw (180pt,60pt) -- (280pt,60pt);
\draw (180pt,30pt) -- (280pt,30pt);
\draw (180pt,0pt) -- (280pt,0pt);
                                  \node (r1) at (240pt,75pt) {$F_{p,q}^{-1}(c)(i+1)$};
\node (l) at (190pt,45pt) {$c(i)$}; \node (r2) at (240pt,45pt) {$c(i+1)$};
                                  \node (r3) at (240pt,15pt) {$F_{p,q}(c)(i+1)$};
\draw[->,thick](r1.west) to [bend right = 45] (l.north);\draw[->,thick](r2) to (l);\draw[->,thick](r3.west) to [bend left = 45] (l.south);
\end{tikzpicture}
\caption{Determination of digits in the space-time diagram of $c$ with respect to $F_{p,q}$.}
\label{expdirs}
\end{figure}
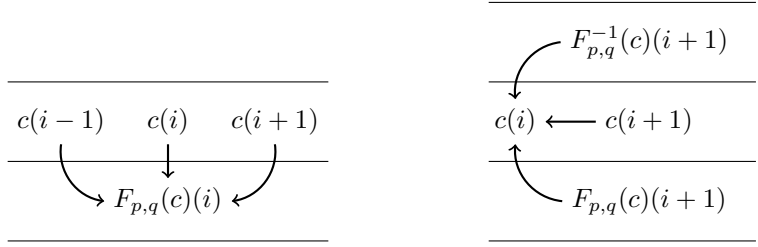

\begin{lmm}\label{g1}If $g_{p,q}(x,z)=g_{p,q}(y,w)$, then $x\equiv y \pmod q$.\end{lmm}
\begin{proof}Let $x=x_1q+x_0$, $y=y_1q+y_0$, $z=z_1q+z_0$ and $w=w_1q+w_0$. Then
\begin{flalign*}
g_{p,q}(x,z)=g_{p,q}(y,w)&\implies x_0p+z_1=y_0p+w_1 \\
&\implies x_0=y_0\implies x\equiv y\pmod q.
\end{flalign*}
\end{proof}

\begin{lmm}\label{g2}$g_{p,q}(x,a)\equiv g_{p,q}(y,a)\pmod q\iff x\equiv y \pmod q$.\end{lmm}
\begin{proof}Let $x=x_1q+x_0$, $y=y_1q+y_0$ and $a=a_1q+a_0$. Then
\begin{flalign*}
g_{p,q}(x,a)\equiv g_{p,q}(y,a)\pmod q&\iff x_0p+a_1\equiv y_0p+a_1\pmod q \\
&\iff x_0=y_0\iff x\equiv y\pmod q 
\end{flalign*}
\end{proof}

\begin{lmm}\label{f1}If $f_{p,q}(x,a,y)=f_{p,q}(z,a,w)$, then $x\equiv z \pmod q$.\end{lmm}
\begin{proof}
\begin{flalign*}
&f_{p,q}(x,a,y)=f_{p,q}(z,a,w) \\
\implies &g_{p,q}(g_{p,q}(x,a),g_{p,q}(a,y))=g_{p,q}(g_{p,q}(z,a),g_{p,q}(a,w)) \\
\overset{L \ref{g1}}{\implies} &g_{p,q}(x,a)\equiv g_{p,q}(z,a)\pmod{q}
\overset{L \ref{g2}}{\implies} x\equiv z\pmod{q}.
\end{flalign*}
\end{proof}

\begin{prpstn}\label{leftdet}For every $c\in A_{pq}^{\Z}$ and for all $k,i\in\Z$, the value of $F_{p,q}^k(c)(i)$ is uniquely determined by the values of $F_{p,q}^{k-1}(c)(i+1)$, $F_{p,q}^k(c)(i+1)$ and $F_{p,q}^{k+1}(c)(i+1)$.\end{prpstn}
\begin{proof}Denote $e=\sigma_{A_{pq}}^i(F_{p,q}^k(c))$. It suffices to show that $e(0)$ is uniquely determined by $F_{q,p}(e)(1)$, $e(1)$ and $F_{p,q}(e)(1)$. Since $F_{p,q}(e)(1)=f_{p,q}(e(0),e(1),e(2))$, by Lemma \ref{f1} $e(1)$ and $F_{p,q}(e)(1)$ determine the value of $e(0)$ modulo $q$ (see Figure \ref{leftdet}, left). Similarly, because $F_{q,p}(e)(1)=f_{q,p}(e(0),e(1),e(2))$, by the same lemma $e(1)$ and $F_{q,p}(e)(1)$ determine the value of $e(0)$ modulo $p$ (Fig. \ref{leftdet}, middle). In total, $F_{q,p}(e)(1)$, $e(1)$ and $F_{p,q}(e)(1)$ determine the value of $e(0)$ both modulo $q$ and modulo $p$ (Fig. \ref{leftdet}, right). Because $e(0)\in A_{pq}$, the value of $e(0)$ is uniquely determined. \end{proof}

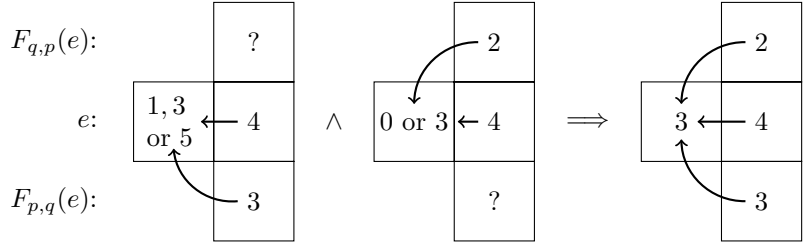
\begin{figure}[h]
\centering
\begin{tikzpicture}
\def\tetris{ rectangle ++(30pt,30pt) ++(-60pt,0pt) rectangle ++(30pt,30pt) ++(0pt,-30pt) rectangle ++(30pt,30pt) ++(-30pt,0pt) rectangle ++(30pt,30pt)}

\node[anchor=east] at (-10pt,75pt) {$F_{q,p}(e)$:};
\node[anchor=east] at (-10pt,45pt) {$e$:};
\node[anchor=east] at (-10pt,15pt) {$F_{p,q}(e)$:};
\draw (30pt,0pt) \tetris;
\node[text width=20pt] at (15pt,45pt) {$1,3$ or $5$};
\node[minimum size=20pt] (l) at (15pt,45pt) {};
\node at (45pt,75pt) {?};
\node at (45pt,45pt) (l1) {$4$};
\node at (45pt,15pt) (l2) {$3$};
\draw[->,thick](l1.west) to (l.east); \draw[->,thick] (l2.west) to [bend left = 45]  (l.south);

\node at (75pt,45pt) {$\land$};

\draw (120pt,0pt) \tetris;
\node (m) at (105pt,45pt) {$0$ or $3$};
\node (m1) at (135pt,75pt) {$2$};
\node (m2) at (135pt,45pt) {$4$};
\node at (135pt,15pt) {?};
\draw[->,thick](m1.west) to [bend right = 45] (m.north); \draw[->,thick] (m2.west) to (m.east);

\node at (170pt,45pt) {$\implies$};

\draw (220pt,0pt) \tetris;
\node (r) at (205pt,45pt) {$3$};
\node (r1) at (235pt,75pt) {$2$};
\node (r2) at (235pt,45pt) {$4$};
\node (r3) at (235pt,15pt) {$3$};
\draw[->,thick](r1.west) to [bend right = 45] (r.north); \draw[->,thick] (r2.west) to (r.east); \draw[->,thick] (r3.west) to [bend left = 45]  (r.south);
\end{tikzpicture}
\caption{The proof of Proposition \ref{leftdet} (here $(p,q)=(3,2)$).}
\label{ldetproof}
\end{figure}

\section{Traces of configurations}

For $\xi\geq0$ we are interested in the values of $\{\xi(p/q)^i\}$ as $i$ ranges over $\N$. In terms of the configuration $\config_{pq}(\xi)$ these correspond to the tails of the configurations $F_{p,q}^i(\config_{pq}(\xi))$, i.e. to the digits $F_{p,q}^i(\config_{pq}(\xi))(j)$ for $j>0$. Partial information on the tails is preserved in the traces of a configuration. In this section we study traces with respect to $F_{p,q}$ to prove in the case $p\geq2q-1$ the existence of small sets $S$ such that $Z_{p/q}(S)$ is non-empty, and then as a corollary for all $p>q>1$.

\begin{dfntn}For any $k\in\Z$, the $k$\emph{-trace} of a configuration $c\in A_{pq}^{\Z}$ (with respect to $F_{p,q}$) is the sequence
\[\tr_{p,q}(c,k)=(F_{p,q}^n(c)(k))_{n\in\Z}.\]
In the special case $k=1$, we denote $\tr_{p,q}(c,1)=\tr_{p,q}(c)$.\end{dfntn}

A $k$-trace of $c$ is simply the sequence of digits in the $k$-th column of the space-time diagram of $c$ with respect to $F_{p,q}$ (see Figure \ref{tr}).

\begin{figure}[h]
\centering
\begin{tikzpicture}
\draw (0pt,100pt) -- (180pt,100pt);
\draw (0pt,80pt) -- (180pt,80pt);
\draw (0pt,60pt) -- (180pt,60pt);
\draw (0pt,40pt) -- (180pt,40pt);
\draw (0pt,20pt) -- (180pt,20pt);
\draw (0pt,0pt) -- (180pt,0pt);
\draw[dashed] (80pt,105pt) -- (80pt,-5pt);
\draw[dashed] (100pt,105pt) -- (100pt,-5pt);
\node[anchor=east] at (-5pt,90pt) {$F_{3,2}^{-2}(c)$};
\node at (10pt,90pt) {$\dots$}; \node at (30pt,90pt) {$5$}; \node at (50pt,90pt) {$4$}; \node at (70pt,90pt) {$0$}; \node at (90pt,90pt) {$1$};
\node at (110pt,90pt) {$5$}; \node at (130pt,90pt) {$3$}; \node at (150pt,90pt) {$4$}; \node at (170pt,90pt) {$\dots$};
\node[anchor=east] at (-5pt,70pt) {$F_{3,2}^{-1}(c)$};
\node at (10pt,70pt) {$\dots$}; \node at (30pt,70pt) {$2$}; \node at (50pt,70pt) {$3$}; \node at (70pt,70pt) {$0$}; \node at (90pt,70pt) {$2$};
\node at (110pt,70pt) {$5$}; \node at (130pt,70pt) {$2$}; \node at (150pt,70pt) {$3$}; \node at (170pt,70pt) {$\dots$};
\node[anchor=east] at (-5pt,50pt) {$c$};
\node at (10pt,50pt) {$\dots$}; \node at (30pt,50pt) {$3$}; \node at (50pt,50pt) {$4$}; \node at (70pt,50pt) {$3$}; \node at (90pt,50pt) {$4$};
\node at (110pt,50pt) {$2$}; \node at (130pt,50pt) {$0$}; \node at (150pt,50pt) {$5$}; \node at (170pt,50pt) {$\dots$};
\node[anchor=east] at (-5pt,30pt) {$F_{3,2}(c)$};
\node at (10pt,30pt) {$\dots$}; \node at (30pt,30pt) {$5$}; \node at (50pt,30pt) {$3$}; \node at (70pt,30pt) {$5$}; \node at (90pt,30pt) {$3$};
\node at (110pt,30pt) {$3$}; \node at (130pt,30pt) {$1$}; \node at (150pt,30pt) {$2$}; \node at (170pt,30pt) {$\dots$};
\node[anchor=east] at (-5pt,10pt) {$F_{3,2}^2(c)$};
\node at (10pt,10pt) {$\dots$}; \node at (30pt,10pt) {$5$}; \node at (50pt,10pt) {$2$}; \node at (70pt,10pt) {$5$}; \node at (90pt,10pt) {$2$};
\node at (110pt,10pt) {$1$}; \node at (130pt,10pt) {$5$}; \node at (150pt,10pt) {$1$}; \node at (170pt,10pt) {$\dots$};
\end{tikzpicture}
\caption{A trace of a configuration.}
\label{tr}
\end{figure}

\begin{dfntn}The set of allowed words of $\tr_{p,q}$ is
\[L(p,q)=\{w\in A_{pq}^*\mid w\text{ is a factor of } \tr_{p,q}(c) \text{ for some } c\in A_{pq}^\Z\},\]
i.e. the set of words that can appear in the columns of space-time diagrams with respect to $F_{p,q}$.\end{dfntn}

The following is a reformulation of Proposition \ref{leftdet} in terms of traces (see Figure \ref{trdet}).

\begin{crllr}\label{leftdetTr}For every $c\in A_{pq}^\Z$ and $k>0$, the values of $\tr_{p,q}(c,k)(i)$ for $-(k-1)\leq i\leq (k-1)$ uniquely determine the values of $c(j)$ for $1\leq j\leq k$. \end{crllr}
\begin{proof}The proof is by induction. The case $k=1$ follows from the fact that $\tr_{p,q}(c,1)(0)=c(1)$. Next assume that the claim holds for some $k>0$ and consider the values of $\tr_{p,q}(c,k+1)(i)$ for $-k\leq i\leq k$. By Proposition \ref{leftdet} these determine $\tr_{p,q}(c,k)(i)$ for $-(k-1)\leq i\leq (k-1)$, which in turn determine $c(j)$ for $1\leq j\leq k$ by the induction hypothesis. The value of $c(k+1)$ is determined by $\tr_{p,q}(c,k+1)(0)=c(k+1)$.\end{proof}

\begin{figure}
\centering
\begin{tikzpicture}
\draw (20pt,80pt) -- (100pt,80pt);
\draw (20pt,60pt) -- (100pt,60pt);
\draw (80pt,140pt) -- (100pt,140pt);
\draw (80pt,0pt) -- (100pt,0pt);

\draw (80pt,145pt) -- (80pt,-5pt);
\draw (100pt,145pt) -- (100pt,-5pt);
\draw (20pt,80pt) -- (20pt,60pt);

\draw[dashed] (20pt,80pt) -- (80pt,140pt);
\draw[dashed] (20pt,60pt) -- (80pt,0pt);

\node[anchor=east] at (-5pt,130pt) {$F_{3,2}^{-3}(c)$};
\node at (90pt,130pt) {$1$};
\node[anchor=east] at (-5pt,110pt) {$F_{3,2}^{-2}(c)$};
\node at (70pt,110pt) {$0$}; \node at (90pt,110pt) {$2$};
\node[anchor=east] at (-5pt,90pt) {$F_{3,2}^{-1}(c)$};
\node at (50pt,90pt) {$4$}; \node at (70pt,90pt) {$3$}; \node at (90pt,90pt) {$3$};
\node[anchor=east] at (-5pt,70pt) {$c$};
\node at (10pt,70pt) {$\dots$}; \node at (30pt,70pt) {$2$}; \node at (50pt,70pt) {$3$}; \node at (70pt,70pt) {$5$}; \node at (90pt,70pt) {$1$};
\node at (110pt,70pt) {$\dots$};
\node[anchor=east] at (-5pt,50pt) {$F_{3,2}(c)$};
\node at (50pt,50pt) {$5$}; \node at (70pt,50pt) {$4$}; \node at (90pt,50pt) {$5$};
\node[anchor=east] at (-5pt,30pt) {$F_{3,2}^2(c)$};
\node at (70pt,30pt) {$4$}; \node at (90pt,30pt) {$2$};
\node[anchor=east] at (-5pt,10pt) {$F_{3,2}^3(c)$};
\node at (90pt,10pt) {$3$};
\end{tikzpicture}
\caption{A trace determining part of the configuration.}
\label{trdet}
\end{figure}

Next we prove an important restriction on the words in the set $L(q,p)$ when $p\geq 2q-1$. Note that the words in $L(q,p)$ are mirror images of the words in $L(p,q)$ (traces with respect to $F_{p,q}$ are read "from bottom to top"). 

\begin{lmm}\label{allowed}Let $p>q\geq 2$ be relatively prime such that $p\geq 2q-1$, and for every $d\in A_q$ let $k_d\in A_p$ and $j_d\in A_q$ be the unique elements such that $k_d q\equiv d\pmod p$ and $k_d q=j_d p+d$. If $wab\in L(q,p)$ for some $w\in A_{pq}^*$, $a,b\in A_{pq}$ and $a\equiv k_d\pmod p$, then $b\equiv j_d\pmod q$.\end{lmm}
\begin{proof}From $wab\in L(q,p)$ it follows that $b=f_{q,p}(x,a,y)$ for some $x,y\in A_{pq}$. Let us write $a=a_1p+a_0$, $y=y_1p+y_0$, $g_{q,p}(x,a)=z=z_1p+z_0$ and $g_{q,p}(a,y)=w=w_1p+w_0$, where $a_0,y_0,z_0,w_0\in A_p$ and $a_1,y_1,z_1,w_1\in A_q$. Here $a_0=k_d$ because $a\equiv k_d\pmod p$ and $w_1=j_d$ because $g_{q,p}(a,y)=k_d q+y_1=j_d p+(d+y_1)$ and $d+y_1\leq(q-1)+(q-1)<p$. Now
\[f_{q,p}(x,a,y)=g_{q,p}(g_{q,p}(x,a),g_{q,p}(a,y))=g_{q,p}(z,w)=z_0 q+j_d,\]
and thus $b\equiv j_d\pmod q$.\end{proof}

Based on the previous lemma, we define a special set of digits
\[D_{p,q}=\{a\in A_{pq}\mid a\equiv k_d\pmod p \text{ for some } d\in A_q\},\]
where the digits $k_d$ are as above.

\begin{xmpl}Consider the case $p=3$ and $q=2$. Then $A_q=\{0,1\}$ and $D_{3,2}=\{0,2,3,5\}$ consists of the elements of $A_6$ which are congruent to either $k_0=0$ or $k_1=2$ $\pmod 3$.\end{xmpl}

\begin{lmm}\label{allowedN}If $p\geq 2q-1$, then $\abs{L(p,q)\cap D_{p,q}^n}\leq q^{n+1}$ for every $n>0$.\end{lmm}
\begin{proof}The proof is by induction. The case $n=1$ is clear because $\abs{D_{p,q}}=q^2$. Next assume that the claim holds for some $n>0$. It is sufficient to compute an upper bound for $\abs{L(q,p)\cap D_{p,q}^{n+1}}$, because the words in $L(p,q)$ are mirror images of the words in $L(q,p)$. If $v\in L(q,p)\cap D_{p,q}^{n+1}$, by the previous lemma it can be written in the form $v=wab$, where $a\equiv k_d\pmod p$ and $b\equiv j_d\pmod q$ for some $d\in A_q$. Because $wa\in L(q,p)\cap D_{p,q}^n$, by the induction hypothesis there are at most $q^{n+1}$ different choices for the word $wa$. Let us fix $wa$ and $d\in A_q$ such that $a\equiv k_d\pmod p$. To prove the claim, it is enough to show that there are at most $q$ choices for the digit $b$.

Let us assume to the contrary that there are distinct digits $b_1,b_2,\dots b_{q+1}\in D_{p,q}$ such that $wab_i\in L(q,p)\cap D_{p,q}^{n+1}$ whenever $1\leq i\leq q+1$. For every $i$ the congruence $b_i\equiv k_{d_i}\pmod p$ holds for some $d_i\in A_q$. By pigeonhole principle we may assume that $d_1=d_2$ and therefore $b_1\equiv k_{d_1}\equiv b_2\pmod p$. Because $wab_1,wab_2\in L(q,p)\cap D_{p,q}^{n+1}$, we also have $b_1\equiv j_d\equiv b_2\pmod q$. Because $b_1,b_2\in A_{pq}$ are congruent both modulo $p$ and modulo $q$, they are equal, contradicting the distinctness of $b_1,b_2,\dots b_{q+1}$.
\end{proof}

As in the introduction, for relatively prime $p>q>1$ and any $S\subseteq [0,1)$ we denote
\[Z_{p/q}(S)=\left\{\xi>0\biggm|\left\{\xi\left(\frac{p}{q}\right)^i\right\}\in S\text{ for every }i\in\N\right\}.\]

In \cite{A} it was proved that if $p,q>1$ are relatively prime integers such that $p>q^2$, then for every $\epsilon>0$ there exists a finite union of intervals $J_{p,q,\epsilon}$ of total length at most $\epsilon$ such that $Z_{p/q}(J_{p,q,\epsilon})\neq\emptyset$. We extend this result to the case $p>q>1$, which in particular covers $p/q=3/2$. The following theorem by Akiyama, Frougny and Sakarovitch is needed.

\begin{sloppypar}
\begin{thrm}[Akiyama, Frougny, Sakarovitch \cite{AFS}]If $p\geq 2q-1$, then $Z_{p/q}(Y_{p,q})\neq\emptyset$, where
\[Y_{p,q}=\bigcup_{d\in A_q}\left[\frac{1}{p}k_d,\frac{1}{p}(k_d+1)\right)\]
and $k_d\in A_p$ are as in Lemma \ref{allowed}.\end{thrm}
\end{sloppypar}

\begin{crllr}If $p\geq 2q-1$, then $Z_{p/q}(X_{p,q})\neq\emptyset$, where
\[X_{p,q}=\bigcup_{a\in D_{p,q}}\left[\frac{1}{pq}a,\frac{1}{pq}(a+1)\right).\]\end{crllr}
\begin{proof}If $\xi\in Z_{p/q}(Y_{p,q})$, then $\xi/q\in Z_{p/q}(X_{p,q})$.
\end{proof}

\begin{thrm}\label{thm1}If $p\geq 2q-1$ and $k>0$, then there exists a finite union of intervals $I_{p,q,k}$ of total length at most $(q/p)^k$ such that $Z_{p/q}(I_{p,q,k})\neq\emptyset$.\end{thrm}
\begin{proof}Let $k>0$ be fixed and choose any $\xi'\in Z_{p/q}(X_{p,q})$, where $X_{p,q}$ is the set in the previous corollary. Let $\xi = \xi'(pq)^{-(k-1)}(p/q)^{k-1}$ and denote $c=\config_{pq}(\xi)$. Based on $c$ we define a collection of words
\[W=\{w=e(1)e(2)\dots e(k)\mid e=F_{p,q}^n(c) \text{ for some } n\in\N\}.\]
The set $W$ determines a finite union of intervals
\[I_{p,q,k}=\bigcup_{w\in W}\left[\real_{pq}(w),\real_{pq}(w)+(pq)^{-k}\right)\]
and $\xi\in Z_{p/q}(I_{p,q,k})$ by the definition of $W$. Each interval in $I_{p,q,k}$ has length $(pq)^{-k}$, so to prove that the total length of $I_{p,q,k}$ is at most $(q/p)^k$ it is sufficient to show that $\abs{W}\leq q^{2k}$.

By the definition of $X_{p,q}$, $\tr_{p,q}(\config_{pq}(\xi'))(i)\in D_{p,q}$ for every $i\geq 0$. For the $k$-trace of $c$
\begin{flalign*}
&\tr_{p,q}(c,k)(i)=\tr_{p,q}(\config_{pq}(\xi'(pq)^{-(k-1)}(p/q)^{k-1}),k)(i) \\
=&\tr_{p,q}(\sigma_{A_{pq}}^{-(k-1)}(F_{p,q}^{k-1}(\config_{pq}(\xi'))),k)(i)=\tr_{p,q}(F_{p,q}^{k-1}(\config_{pq}(\xi')),1)(i) \\
=&\tr_{p,q}(\config_{pq}(\xi'))(i+(k-1))\text{ for every }i\in\N,
\end{flalign*}
from which it follows that $\tr_{p,q}(c,k)(i)\in D_{p,q}$ for every $i\geq-(k-1)$. Thus, the words in the set
\[V=\{\tr_{p,q}(F_{p,q}^n(c),k)(-(k-1))\dots \tr_{p,q}(F_{p,q}^n(c),k)(k-1)\mid n\in\N\}\]
also belong to $L(p,q)\cap D_{p,q}^{2k-1}$, and by Corollary \ref{leftdetTr} and Lemma \ref{allowedN}
\[\abs{W}\leq\abs{V}\leq\abs{L(p,q)\cap D_{p,q}^{2k-1}}\leq q^{2k}.\]
\end{proof}

\begin{rmrk}The set $I_{p,q,k}$ constructed in the proof of the previous theorem is a union of $q^{2k}$ intervals, each of which is of length $(pq)^{-k}$.\end{rmrk}

\begin{crllr}\label{cor}If $p>q>1$ and $\epsilon>0$, then there exists a finite union of intervals $J_{p,q,\epsilon}$ of total length at most $\epsilon$ such that $Z_{p/q}(J_{p,q,\epsilon})\neq\emptyset$.\end{crllr}
\begin{proof}Choose some $n>0$ such that $p^n\geq 2q^n-1$. Then by the previous theorem there exists a finite union of intervals $I_0$ of total length at most $\eta=\epsilon(p-1)/(p^n-1)$ such that $Z_{p^n/q^n}(I_0)\neq\emptyset$. For $0<i<n$ define inductively
\[I_i=\left\{\left\{\xi\frac{p}{q}\right\}\in[0,1)\biggm|\xi\geq 0 \text{ and }\left\{\xi\right\}\in I_{i-1}\right\},\]
each of which is a finite union of intervals of total length at most $p^i\eta$. Then $J_{p,q,\epsilon}=\bigcup_{i=0}^{n-1}I_i$ is a finite union of intervals of total length at most
\[\sum_{i=0}^{n-1}(p^i)\eta=\frac{p^n-1}{p-1}\eta=\epsilon\]
and $Z_{p/q}(J_{p,q,\epsilon})\supseteq Z_{p^k/q^k}(I_0)\neq\emptyset$.
\end{proof}

\section{Ergodicity of $F_{p,q}$}

In this section we study the measure theoretical properties of $F_{p,q}$ to prove the existence of large sets $S$ such that $Z_{p/q}(S)$ is empty.

\begin{sloppypar}
\begin{dfntn}A CA function $F:A^\Z\to A^\Z$ is \emph{measure preserving} if $\mu(F^{-1}(S))=\mu(S)$ for every $S\in\Sigma(\mathcal{C})$.
\end{dfntn}
\end{sloppypar}

\begin{dfntn}A measure preserving CA function $F:A^\Z\to A^\Z$ is \emph{ergodic} if for every $S\in \Sigma(\mathcal{C})$ with $F^{-1}(S)=S$ either $\mu(S)=0$ or $\mu(S)=1$.\end{dfntn}

The next lemma is a special case of a well known measure theoretical result (see e.g. Theorem 2.18 in \cite{R}):
\begin{lmm}\label{cover}For every $S\in\Sigma(\mathcal{C})$ and $\epsilon>0$ there is an open set $U\subseteq A^{\Z}$ such that $S\subseteq U$ and $\mu(U\setminus S)<\epsilon$.\end{lmm}

\begin{lmm}\label{ergcover}If $F:A^\Z\to A^\Z$ is an ergodic CA, then for every $\epsilon>0$ there is a finite collection of cylinders $\{U_i\}_{i\in I}$ such that
$\mu(\bigcup_{i\in I} U_i)<\epsilon$ and
\[\left\{c\in A^\Z\mid F^t(c)\in\bigcup_{i\in I} U_i\text{ for some }t\in\N\right\}=A^\Z.\]
\end{lmm}
\begin{proof}
Let $C\in\mathcal{C}$ be such that $0<\mu(C)<\epsilon/2$. By continuity of $F$, $B=\bigcup_{t\in\N}F^{-t}(C)$ is open and $\mu(B)=1$ by ergodicity of $F$ (see Theorem 1.5 in \cite{Wa}). Equivalently, $B'=A^{\Z}\setminus B$ is closed (and compact) and $\mu(B')=0$. Let $V$ be an open set such that $B'\subseteq V$ and $\mu(V)<\epsilon/2$: such a set exists by Lemma \ref{cover}. Because $\mathcal{C}$ is a basis of $\mathcal{T}$, there is a collection of cylinders $\{V_i\}_{i\in J}$ such that $V=\bigcup_{i\in J}V_i$. By compactness of $B'$ there is a finite set $I'\subseteq J$ such that $B'\subseteq \bigcup_{i\in I'}V_i$. Now $\{U_i\}_{i\in I}=\{C\}\cup \{V_i\}_{i\in I'}$ is a finite collection of cylinders such that $\mu(\bigcup_{i\in I} U_i)<\epsilon$ and
\[\left\{c\in A^\Z\mid F^t(c)\in\bigcup_{i\in I} U_i\text{ for some }t\in\N\right\}\supseteq B\cup\bigcup_{i\in I'}V_i\supseteq B\cup B'=A^\Z.\]
\end{proof}

To apply this lemma in our setup, we need to show that $F_{p,q}$ is ergodic for $p>q>1$. In fact, it turns out that a stronger result holds.

\begin{dfntn}A measure preserving CA function $F:A^\Z\to A^\Z$ is \emph{strongly mixing} if 
\[\lim_{t\to\infty}\mu(F^{-t}(U)\cap V)=\mu(U)\mu(V)\]
for every $U,V\in\Sigma(\mathcal{C})$.\end{dfntn}

We will prove that $F_{p,q}$ is strongly mixing. For the statement of the following lemmas, we define a function $\integ:A_{pq}^+\to\N$ by
\[\integ(w(1)w(2)\dots w(k))=\sum_{i=0}^{k-1}w(k-i)(pq)^i,\]
i.e. $\integ(w)$ is the integer having $w$ as a base $pq$ representation.

\begin{lmm}Let $w_1,w_2\in A_{pq}^k$ for some $k\geq 2$ and let $t>0$ be a natural number. Then
\begin{enumerate}
\item $\integ(w_1)<q^t \implies \integ(G_{p,q}(w_1))<q^{t-1}$ and
\item $\integ(w_2)\equiv \integ(w_1)+q^t\pmod{(pq)^k} \implies \integ(G_{p,q}(w_2))\equiv \integ(G_{p,q}(w_1))+q^{t-1} \pmod{(pq)^{k-1}}$.
\end{enumerate}
\end{lmm}
\begin{proof}
Let $c_i\in A_{pq}^\Z$ ($i=1,2$) be such that $c_i(-(k-1))c_i(-(k-1)+1)\dots c_i(0)=w_i$ and $c_i(j)=0$ for $j<-(k-1)$ and $j>0$. From this definition of $c_i$ it follows that $\integ(w_i)=\real_{pq}(c_i)$. Denote $e_i=G_{p,q}(c_i)$. We have
\[\sum_{j=-\infty}^{\infty}e_i(-j)(pq)^j=\real_{pq}(e_i)=p\real_{pq}(c_i)=p\integ(w_i)\]
and
\begin{flalign*}
\integ(G_{p,q}(w_i))&=\integ(e_i(-(k-1))\dots e_i(-1)) \\
&=\sum_{j=1}^{k-1}e_i(-j)(pq)^{j-1}\equiv\lfloor \integ(w_i)/q\rfloor\pmod{(pq)^{k-1}}.
\end{flalign*}
Also note that $\integ(G_{p,q}(w_i))<(pq)^{k-1}$.

For the proof of the first part, assume that $\integ(w_1)<q^t$. Combining this with the observations above yields $\integ(G_{p,q}(w_i))\leq\lfloor \integ(w_i)/q\rfloor<q^{t-1}$.

\begin{sloppypar}
For the proof of the second part, assume that $\integ(w_2)\equiv \integ(w_1)+q^t\pmod{(pq)^k}$. Then there exists $n\in\Z$ such that $\integ(w_2)=\integ(w_1)+q^t+n(pq)^k$ and
\end{sloppypar}
\begin{flalign*}
\integ(G_{p,q}(w_2))&\equiv \lfloor \integ(w_2)/q\rfloor\equiv\lfloor \integ(w_1)/q\rfloor+q^{t-1}+np(pq)^{k-1} \\
&\equiv\lfloor \integ(w_1)/q\rfloor+q^{t-1}\equiv \integ(G_{p,q}(w_1))+q^{t-1}\pmod{(pq)^{k-1}}.
\end{flalign*}
\end{proof}

\begin{lmm}\label{odometer}Let $t>0$ and $w_1,w_2\in A_{pq}^k$ for some $k\geq 2t+1$. Then
\begin{enumerate}
\item $\integ(w_1)<q^{2t} \implies \integ(F_{p,q}^t(w_1))=0$ and
\item $\integ(w_2)\equiv \integ(w_1)+q^{2t}\pmod{(pq)^k} \implies \integ(F_{p,q}^t(w_2))\equiv \integ(F_{p,q}^t(w_1))+1 \pmod{(pq)^{k-2t}}$.
\end{enumerate}
\end{lmm}
\begin{proof}
First note that $F_{p,q}(w)=G_{p,q}^2(w)$ for every $w\in A_{pq}^*$ such that $\abs{w}\geq 3$, because $F_{p,q}=\sigma_{A_{pq}}^{-1}\circ G_{p,q}\circ G_{p,q}$. Then both claims follow by repeated application of the previous lemma.
\end{proof}

The content of Lemma \ref{odometer} is as follows. Assume that $\{w_i\}_{i=0}^{(pq)^k-1}$ is the enumeration of all the words in $A_{pq}^k$ in the lexicographical order, meaning that $w_0=00\dots 00$, $w_1=00\dots 01$, $w_2=00\dots 02$ and so on. Then let $i$ run through all the integers between $0$ and $(pq)^k-1$. For the first $q^{2t}$ values of $i$ we have $F_{p,q}^t(w_i)=00\dots 00$, for the next $q^{2t}$ values of $i$ we have $F_{p,q}^t(w_i)=00\dots 01$, and for the following $q^{2t}$ values of $i$ we have $F_{p,q}^t(w_i)=00\dots 02$. Eventually, as $i$ is incremented from $q^{2t}(pq)^{k-2t}-1$ to $q^{2t}(pq)^{k-2t}$, the word $F_{p,q}^t(w_i)$ loops from $(pq-1)(pq-1)\dots (pq-1)(pq-1)$ back to $00\dots 00$.

\begin{thrm}If $p>q>1$, then $F_{p,q}$ is strongly mixing.\end{thrm}
\begin{proof} Firstly, because $F_{p,q}$ is surjective, the fact that $F_{p,q}$ is measure preserving follows from Theorem 5.4 in \cite{H}. Then, by Theorem 1.17 in \cite{Wa} it is sufficient to verify the condition
\[\lim_{t\to\infty}\mu(F_{p,q}^{-t}(C_1)\cap C_2)=\mu(C_1)\mu(C_2)\]
for every $C_1,C_2\in\mathcal{C}$. Without loss of generality we may consider cylinders $C_1=\cyl(v_1,0)$ and $C_2=\cyl(v_2,i)$. Denote $l_1=\abs{v_1}$, $l_2=\abs{v_2}$ and let $t\geq i+l_2$ be a natural number.

Consider an arbitrary word $w\in A_{pq}^{2t+l_1}$ and its decomposition $w=w_1w_2w_3$, where $w_1\in A_{pq}^{t+i}$, $w_2\in A_{pq}^{l_2}$ and $w_3\in A_{pq}^{t+l_1-i-l_2}$. The following conditions may or may not be satisfied by $w$ (see Figure \ref{sm}):
\begin{enumerate}
\item $F_{p,q}^t(w)=v_1$
\item $w_2=v_2$.
\end{enumerate}
\begin{figure}[h]
\centering
\begin{tikzpicture}
\draw (0pt,100pt) -- (240pt,100pt);
\draw (0pt,100pt) -- (0pt,80pt); \node at (50pt,87pt) {$w_1$}; \draw (100pt,100pt) -- (100pt,80pt); \node at (125pt,90pt) {$w_2\overset{?}{=}v_2$}; \draw (150pt,100pt) -- (150pt,80pt); \node at (195pt,87pt) {$w_3$}; \draw (240pt,100pt) -- (240pt,80pt);
\draw (0pt,80pt) -- (240pt,80pt);

\draw (80pt,20pt) -- (160pt,20pt);
\draw (80pt,20pt) -- (80pt,0pt); \node at (120pt,10pt) {$F_{p,q}^t(w)\overset{?}{=}v_1$}; \draw (160pt,20pt) -- (160pt,0pt);
\draw (80pt,0pt) -- (160pt,0pt);

\draw[dashed] (0pt,80pt) -- (80pt,0pt);
\draw[dashed] (240pt,80pt) -- (160pt,0pt);

\draw[dashed] (80pt,80pt) -- (80pt,20pt);
\draw[dashed] (160pt,80pt) -- (160pt,20pt);

\draw[decorate,decoration={brace,mirror,amplitude=5pt}] (0pt,80pt) -- (80pt,80pt); \node at (40pt,70pt){$t$};
\draw[decorate,decoration={brace,mirror,amplitude=5pt}] (100pt,80pt) -- (150pt,80pt); \node at (125pt,70pt){$l_2$};
\draw[decorate,decoration={brace,mirror,amplitude=5pt}] (160pt,80pt) -- (240pt,80pt); \node at (200pt,70pt){$t$};
\draw[decorate,decoration={brace,amplitude=5pt}] (80pt,20pt) -- (160pt,20pt); \node at (120pt,30pt){$l_1$};

\draw (80pt,100pt) -- (80pt,105pt); \node at (80pt,110pt) {$0$};
\draw (100pt,100pt) -- (100pt,105pt); \node at (100pt,110pt) {$i$};

\draw[->] (245pt,90pt) .. controls (250pt,50pt) .. (245pt,10pt); \node at (260pt,50pt) {$F_{p,q}^t$};

\end{tikzpicture}
\caption{Relations between the words $v_1$, $v_2$ and $w_1w_2w_3$.}
\label{sm}
\end{figure}
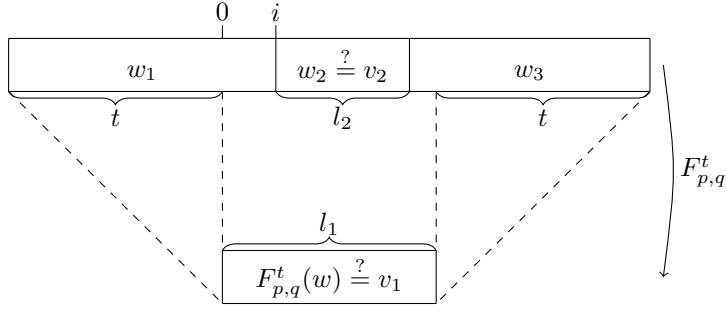
Note that if $w$ satisfies condition $(1)$, then $F_{p,q}^t(\cyl(w,-t))\subseteq C_1$, and otherwise $F_{p,q}^t(\cyl(w,-t))\cap C_1=\emptyset$. Also, if $w$ satisfies condition $(2)$, then $\cyl(w,-t)\subseteq C_2$, and otherwise $\cyl(w,-t)\cap C_2=\emptyset$. Let $W_t\subseteq A_{pq}^{2t+l_1}$ be the collection of those words $w$ that satisfy both conditions. It follows that

\[\mu(F_{p,q}^{-t}(C_1)\cap C_2)=\mu\left(\bigcup_{w\in W_t}\cyl(w,-t)\right)=\abs{W_t}(pq)^{-(2t+l_1)}.\]

\begin{sloppypar}
Next, we estimate the number of words $w=w_1w_2w_3$ in $W_t$. In any case, to satisfy condition $(2)$, $w_2$ must equal $v_2$. Then, for any of the $(pq)^{t+i}$ choices of $w_1$, the number of choices for $w_3$ that satisfy condition $(1)$ is between $(pq)^{t+l_1-i-l_2}/(pq)^{l_1}-q^{2t}$ and $(pq)^{t+l_1-i-l_2}/(pq)^{l_1}+q^{2t}$ by Lemma \ref{odometer} (and the paragraph following it). Thus,
\end{sloppypar}
\begin{flalign*}
&\left((pq)^{t-i-l_2}-q^{2t}\right)(pq)^{t+i}(pq)^{-(2t+l_1)}\leq \mu(F_{p,q}^{-t}(C_1)\cap C_2) \\
\leq& \left((pq)^{t-i-l_2}+q^{2t}\right)(pq)^{t+i}(pq)^{-(2t+l_1)},
\end{flalign*}
and as $t$ tends to infinity,
\[\lim_{t\to\infty}\mu(F_{p,q}^{-t}(C_1)\cap C_2)=(pq)^{-l_1-l_2}=\mu(C_1)\mu(C_2).\]
\end{proof}

\begin{thrm}\label{thm2}If $p>q>1$ and $\epsilon>0$, then there exists a finite union of intervals $K_{p,q,\epsilon}$ of total length at least $1-\epsilon$ such that $Z_{p/q}(K_{p,q,\epsilon})=\emptyset$.\end{thrm}
\begin{proof}
The previous theorem implies that $F_{p,q}$ is ergodic: if $S\in \Sigma(\mathcal{C})$ is such that $F_{p,q}^{-1}(S)=S$, then
\[\mu(S)=\lim_{t\to\infty}\mu(F_{p,q}^{-t}(S)\cap S)=\mu(S)\mu(S),\]
which means that $\mu(S)=0$ or $\mu(S)=1$.

Since $F_{p,q}$ is ergodic, by Lemma \ref{ergcover} there is a finite collection of cylinders $\{U_i\}_{i\in I}$ such that
$\mu(\bigcup_{i\in I} U_i)<\epsilon$ and
\[\left\{c\in A_{pq}^\Z\mid F_{p,q}^t(c)\in\bigcup_{i\in I} U_i\text{ for some }t\in\N\right\}=A_{pq}^\Z.\]
Without loss of generality we may assume that for every $i\in I$, $U_i=\cyl(w_i,1)$ and $w_i\in A_{pq}^k$ for a fixed $k>0$. Consider the collection of words $W=A_{pq}^k\setminus\{w_i\}_{i\in I}$ and define
\[K_{p,q,\epsilon}=\bigcup_{v\in W}\left[\real_{pq}(v),\real_{pq}(v)+(pq)^{-k}\right).\]
The set $K_{p,q,\epsilon}$ has total length
\[\frac{\abs{W}}{(pq)^k}=1-\frac{\abs{I}}{(pq)^k}=1-\mu\left(\bigcup_{i\in I} U_i\right)\geq 1-\epsilon.\]
Now let $\xi>0$ be arbitrary and denote $c=\config_{pq}(\xi)$. There exists a $t\in\N$ such that $F_{p,q}^t(c)\in\bigcup_{i\in I} U_i$, and equivalently, $F_{p,q}^t(c)\notin\bigcup_{v\in W}(\cyl(v,1))$. This means that $\{\xi(p/q)^t\}\notin K_{p,q,\epsilon}$, and therefore $Z_{p/q}(K_{p,q,\epsilon})=\emptyset$.
\end{proof}

\section{Conclusions}
We have shown in Theorem \ref{thm1} and Corollary \ref{cor} that for $p>q>1$ and $\epsilon>0$ there exists a finite union of intervals $J_{p,q,\epsilon}$ of total length at most $\epsilon$ such that $Z_{p/q}(J_{p,q,\epsilon})\neq\emptyset$. Moreover, by following the proof of this result, it is possible (at least in principle) to explicitly construct the set $J_{p,q,\epsilon}$ for any given $\epsilon$. We have also shown in Theorem \ref{thm2} that for $p>q>1$ and $\epsilon>0$ there exists a finite union of intervals $K_{p,q,\epsilon}$ of total length at least $1-\epsilon$ such that $Z_{p/q}(K_{p,q,\epsilon})=\emptyset$. The proof of this theorem is non-constructive.

\begin{prblm}Assume that $p>q>1$. Is it possible to construct explicitly for every $\epsilon>0$ a finite union of intervals $S$ such that the total length of $S$ is at least $1-\epsilon$ and $Z_{p/q}(S)=\emptyset$?\end{prblm}

\end{document}